\documentclass[11pt]{article}
\usepackage{amssymb}
\usepackage{latexsym}
\usepackage{amsmath}
\usepackage{amsthm}
\usepackage{amscd}
\usepackage{graphicx}
\newtheorem{theorem}{Theorem}
\newtheorem{proposition}[theorem]{Proposition}

\newtheorem{lemma}[theorem]{Lemma}
\newtheorem{lemma-definition}[theorem]{Lemma-Definition}

\newtheorem{claim}[theorem]{Claim}

\newtheorem{fact}[theorem]{Fact}

\theoremstyle{definition}

\newcommand{\eqdef}{\;{:=}\;}

\newcommand{\C}{{\mathbb C}}

\newcommand{\R}{{\mathbb R}}

\newcommand{\Z}{{\mathbb Z}}

\newcommand{\op}{\operatorname}

\newcommand{\bpm}{\begin{pmatrix}}
\newcommand{\epm}{\end{pmatrix}}

\renewcommand{\epsilon}{\varepsilon}

\begin{document}

\setcounter{tocdepth}{2}

\title{Symplectic embeddings of products}

\author{D. Cristofaro-Gardiner \and R. Hind}

\date{\today}

\maketitle

\begin{abstract}

McDuff and Schlenk determined when a four-dimensional ellipsoid can be symplectically embedded into a four-dimensional ball, and found that when the ellipsoid is close to round, the answer is given by an “infinite staircase” determined by the odd-index Fibonacci numbers.  We show that this result still holds in higher dimensions when we ``stabilize" the embedding problem.

\end{abstract}

\begin{section}{Introduction.}

Recent years have seen much progress on the symplectic embedding problem, particularly in dimension $4$. A highlight was McDuff and Schlenk's classification of embeddings of $4$-dimensional ellipsoids into balls \cite{ms}. To state their result let us first introduce some notation that we use throughout this paper.

Consider Euclidean space $\R^{2N}$, with coordinates $x_j,y_j$, $1 \le j \le N$, equipped with its standard symplectic form $\omega = \sum_{j=1}^N dx_j \wedge dy_j$. Often it is convenient to identify $\R^{2N}$ with $\C^N$ by setting $z_j = x_j + i y_j$.   Now define the {\em symplectic ellipsoid}:
\[E(a_1, \dots ,a_N) = \left\{\sum_j \frac{\pi |z_j|^2}{a_j} \le 1\right\}.\]
These are subsets of $\C^N$ and so inherit the symplectic structure. A {\em ball} of capacity $R$ is simply an ellipsoid $B^{2N}(R) = E(R, \dots ,R)$; it is also convenient to write $\lambda E(a_1, \dots ,a_N)$ for $E(\lambda a_1, \dots ,\lambda a_N)$.

Symplectic ellipsoids provide a very fruitful source of examples for studying symplectic embedding problems.  Indeed, it is currently very much unknown when precisely one $2n$-dimensional symplectic ellipsoid embeds into another.  In \cite{ms}, McDuff and Schlenk completely determined the function $$c_B(x)=\inf \{R \hspace{2 mm} | \hspace{2 mm} E(1,x) \hookrightarrow B^4(R)\},$$ where $E(1,x) \hookrightarrow B^4(R)$ denotes a symplectic embedding. By reordering factors and appealing to scaling, the function $c_B(x)$ for $x \ge 1$ completely determines when a four-dimensional ellipsoid can be symplectically embedded into a four-dimensional ball.  It turns out that $c_B(x)$ is especially interesting when $1 \le x \le \tau^4$, where $\tau = (1+\sqrt{5})/2$ is the golden ratio. Here the function is an infinite staircase defined by ratios of odd Fibonacci numbers, as we will describe below.

In the current paper, we fix a dimension $2N \ge 6$ and consider the ``stabilized" version of $c_B$ given by
$$f(x) = \inf \{R| E(1,x) \times \R^{2(N-2)} \hookrightarrow B^4(R) \times \R^{2(N-2)}\}.$$  Our motivation is to understand to what extent $4$-dimensional features persist for higher dimensional embedding problems.  As $4$-dimensional embeddings $E(1,x) \hookrightarrow B^4(R)$ induce (by taking a product with the identity) high dimensional embeddings $E(1,x) \times \R^{2(N-2)} \hookrightarrow B^4(R) \times \R^{2(N-2)}$, we see immediately that
\begin{equation}
\label{eqn:upperbound}
f(x) \le c_B(x).
\end{equation}
One might guess that in fact $f(x)=c_B(x)$. This however is not necessarily the case -- there are no volume obstructions to embeddings into a product $B^4(R) \times \R^{2(N-2)},$ and the methods applied by McDuff and Schlenk are explicitly $4$-dimensional, relying on Seiberg-Witten theory and special properties of holomorphic curves in dimension $4$.  In fact, the following ``folding" construction in \cite{hind} shows that there is significant flexibility for the stabilized problem:

 \begin{theorem} (Hind, \cite{hind}) For any $S,x \ge 1$ and $\epsilon>0$ there exists a symplectic folding mapping $E(1,x,S \dots, S) \hookrightarrow B^4(\frac{3x}{x+1}+\epsilon) \times \R^{2(N-2)}$.
\end{theorem}

Work of Pelayo and Ng\d{o}c, see \cite{pelngo1}, Theorem $4.1$, implies that these embeddings can be extended to $E(1,x) \times \R^{2(N-2)}$ and so we have that $f(x) \le \frac{3x}{x+1}$.  On the other hand, the four-dimensional volume obstruction implies that $c_B(x) \ge \sqrt{x}$.  It follows that we must have $f(x) < c_B(x)$ when $x > \tau^4$. Nevertheless, our main theorem states that the infinite staircase does in fact persist:

\begin{theorem}\label{main1} If $1 \le x \le \tau^4$ then $f(x)=c_B(x)$.
\end{theorem}

Thus, Theorem~\ref{main1} in combination with the discussion above implies that the optimal embedding for the stabilized problem is given by a product precisely up until $\tau^4$, after which strictly better constructions are available. It is an interesting open question to determine $f(x)$ for $x > \tau^4$.

\newpage

{\bf Sketch of methods.}

\vspace{3 mm}
To describe our approach to Theorem \ref{main1} we recall the structure of McDuff and Schlenk's infinite staircase.

Let $g_0=1$ and $g_n$ for $n\ge 1$ be the $n$th odd Fibonacci number. Thus $\{g_n\}_{n=0}^{\infty}$ is the sequence beginning $1, 1, 2, 5, 13, 34, \dots$. Then we can define sequences $\{a_n\}_{n=0}^{\infty}$ and $\{b_n\}_{n=0}^{\infty}$ by $a_n = \left( \frac{g_{n+1}}{g_n} \right)^2$ and $b_n=\frac{g_{n+2}}{g_n}$. We have $\lim_{n \to \infty} b_n = \tau^4 = \frac{7+3\sqrt{5}}{2}$. Given this, Theorem $1.1.2$ in \cite{ms} says the following.

\begin{theorem} \label{cfn} (McDuff-Schlenk, \cite{ms}, Theorem $1.1.2$) On the interval $1 \le x \le \tau^4$ the function $c_B(x)$ is linear on the intervals $[a_n, b_n]$ and constant on the intervals $[b_n, a_{n+1}]$. We have $c_B(a_n)=\frac{g_{n+1}}{g_n}$ and $c_B(b_n) = c_B(a_{n+1}) = \frac{g_{n+2}}{g_{n+1}}$.
\end{theorem}

The following lemma follows directly from \cite{ms}, Lemma $1.1.1$.

\begin{lemma}\label{lem:mcslemma} If $f(b_n) = c_B(b_n)$ for all $n$, then $f(x) = c_B(x)$ for all $x \le \tau^4$.
\end{lemma}

\begin{proof} Since $f$ is nondecreasing, the hypothesis together with \eqref{eqn:upperbound} imply that $f(x)=c_B(x)$ on the intervals $[b_n, a_{n+1}]$.

Next, as in \cite{ms} Lemma $1.1.1$, we observe that
\begin{equation}
\label{eqn:sublinearity}
f(\lambda x) \le \lambda f(x)
\end{equation}
for $\lambda \ge 1$. Let the interval $I=[a_n, b_n]$ for some $n$. As the graph of $c_B |_I$ lies on a line through the origin, and since from the above $f(x)$ coincides with $c_B(x)$ at the endpoints of $I$, the observation \eqref{eqn:sublinearity} implies that in fact $f=c_B$ on the whole interval and this completes the proof.
\end{proof}

The proof of Theorem \ref{main1} thus reduces to showing that $f(b_n)= \frac{g_{n+2}}{g_{n+1}}$ for all $n \ge 1$ and we will prove this by studying holomorphic curves in symplectic cobordisms.

\vspace{3 mm}

{\bf Outline of the paper.}

\vspace{3 mm}

Our argument combines ideas from \cite{hk} with some techniques involving embedded contact homology.  The details are as follows.

In section \ref{secn2} we consider an embedding $\phi:E(1,b_n+\epsilon) \hookrightarrow \op{int}(E(c,c+\epsilon))$ where $c$ is slightly larger than $c_B(b_n)$ and $\epsilon > 0$ is small. We look at holomorphic curves in the completion of the cobordism $E(c,c+\epsilon) \setminus \phi(\op{int}(E(1,b_n+\epsilon)))$ and establish the nontriviality of a certain moduli space of curves with $g_{n+1}$ positive ends and a single negative end. This relies heavily on the machinery of embedded contact homology.

In section \ref{secn3} we consider a product embedding $\tilde{\phi}:E(1,b_n+\epsilon,S, \dots ,S) \hookrightarrow E(c,c+\epsilon)\times \R^{2(N-2)}$ and a corresponding $2N$ dimensional cobordism. For a suitable choice of almost-complex structure the curves constructed in dimension $4$ imply that a corresponding moduli space of curves in the high dimensional cobordism is also nontrivial. We proceed to show that in fact it is also nontrivial as a cobordism class.

In section \ref{secn4} we prove a compactness theorem showing that the cobordism class of the moduli space studied in section \ref{secn3} is the same for all embeddings $\tilde{\psi}:\lambda E(1,b_n+\epsilon,S, \dots ,S) \hookrightarrow E(c,c+\epsilon)\times \R^{2(N-2)}$ with $\lambda>0$, and in particular is always nontrivial. As holomorphic curves have positive area this readily implies Theorem \ref{main1} as we make precise in section \ref{secn5}.

\vspace{3 mm}

{\bf Acknowledgements}

\vspace{3 mm}

This paper arose from conversations between the two authors at the ``Transversality in contact homology" conference organized by the American Institute of Mathematics (AIM).  We thank AIM for their hospitality.  The first author is partially supported by NSF grant DMS-1402200. The second author is partially supported by grant \# 317510 from the Simons Foundation.

\end{section}

\begin{section}{Four dimensional cobordisms.}\label{secn2}

\subsection{The main proposition}
\label{sec:prpn}

Fix a sufficiently small irrational $\epsilon > 0$, and consider the four-dimensional symplectic ellipsoids
\[ E_1 \eqdef \lambda E\left(\frac{g_{n+2}}{g_{n+1}},\frac{g_{n+2}}{g_{n+1}}+\epsilon\right), \quad E_2 \eqdef E\left(1,\frac{g_{n+2}}{g_n}+\epsilon\right), \quad .\]
where $\lambda > 1$ is some real number close to $1$.  As with any irrational ellipsoid $E(a,b)$, these have a natural contact form with exactly two Reeb orbits, one of action $a$ and the other of action $b$.  Here, the {\em action} of a Reeb orbit $\gamma$ is defined by
\[ \mathcal{A}(\gamma) = \int_{\gamma} \mu\]
where $\mu$ is a Liouville form on $\C^2$.

Let $\alpha_1$ denote the short Reeb orbit on $\partial E_1$ and let $\alpha_2$ denote the long orbit; define $\beta_1$ and $\beta_2$ on $\partial E_2$ analogously.
By an {\em orbit set} we mean a finite collection of distinct embedded Reeb orbits with multiplicities, which we write with multiplicative notation. For example $\alpha = \alpha_1^k \alpha_2^l$ is an orbit set in $\partial E_1$ for any $k,l \ge 0$.
It is useful to define the action of an orbit set by
\[\mathcal{A}\left(\sum_i \alpha_i^{m_i}\right) = \sum_i m_i \mathcal{A}(\alpha_i).\]

Recall that by \cite{ms}, see Theorem \ref{cfn} above, there is a symplectic embedding $\Psi: E_2 \to \op{int}(E_1)$ for any $\lambda>1$.  Choose such a $\lambda$ close to $1$, let $X$ denote the symplectic cobordism $E_1 \setminus \Psi(E_2)$, and let $\overline{X}$ denote the symplectic completion of $X$ (see for instance \cite{echlecture}, section $5.5$).  Let $J$ denote a ``cobordism admissible" (in the sense of \cite{echlecture} again for example) almost complex structure on $\overline{X},$ and for orbit sets $\alpha$ and $\beta$, let $\mathcal{M}(\alpha,\beta)$ denote the moduli space of $J$-holomorphic curves in $\overline{X}$ asymptotic to an orbit set $\alpha$ at $+\infty$ and $\beta$ at $-\infty$.  Saying that a holomorphic curve is {\em asymptotic} to $\alpha = \alpha_1^k \alpha_2^l$ means that its positive ends cover $\alpha_1$ with total multiplicity $k$ and $\alpha_2$ with total multiplicity $l$, see eg \cite{echlecture} for more details.

The goal of this section is to prove the following.

\begin{proposition}
\label{prop:mainprop}

For any $n \ge 0$, if $\epsilon$ is sufficiently small and $\lambda$ is sufficiently close to $1$, then there is a connected embedded $J$-holomorphic curve $C \in \mathcal{M}(\alpha_2^{g_{n+1}},\beta_1^{g_{n+2}})$.  The curve $C$ has genus $0$, $g_{n+1}$ positive ends, and one negative end. In other words, each positive end is asymptotic to $\alpha_2$ and the negative end is asymptotic to the degree $g_{n+2}$ cover of $\beta_1$.
\end{proposition}

To put this slightly differently, let us define $\mathcal{M}_0$ to be the moduli space of $J$-holomorphic curves in $\overline{X}$ with $g_{n+1}$ positive ends asymptotic to $\alpha_2$ and one negative end asymptotic to $\beta_1^{g_{n+2}}$. Proposition \ref{prop:mainprop} says that $\mathcal{M}_0$ is nonempty.

\subsection{ECH Preliminaries}
\label{sec:prelim}

Let $(Y,\mu)$ be a closed three-manifold with a nondegenerate contact form.  The {\em embedded contact homology} of $Y$ (with $\Z/2$ coefficients), $ECH_*(Y,\mu)$, is the homology of a chain complex $ECC_*(Y,\mu)$.  This chain complex is freely generated over $\Z/2$ by orbit sets, where the definition of orbit set was given in section \ref{sec:prpn}.  The orbit sets are required to be {\em admissible}.   This means that $m_i$ is equal to $1$ whenever $\alpha_i$ is hyperbolic.  The chain complex differential $d$ is defined by counting {\em ECH index} $1$ ``$J$-holomorphic currents" in $\mathbb{R} \times Y$, for admissible $J$.  Specifically, the coefficient $\langle d\alpha, \beta \rangle$ is a mod $2$ count of ECH index $1$ $J$-holomorphic currents, modulo translation in the $\R$-direction, that are asymptotic to $\alpha$ at $+\infty$ and asymptotic to $\beta$ at $-\infty$; for the definiton of asymptotic in this context, see the previous section.  By a {\em holomorphic current}, we mean a finite set $\lbrace (C_i,m_i) \rbrace$, where the $C_i$ are irreducible \footnote{We call a somewhere injective curve {\em irreducible} if its domain is connected.} somewhere injective $J$-holomorphic curves in $\R \times Y$ and the $m_i$ are positive integers.  Two $J$-holomorphic currents are declared equivalent if they are equivalent as currents.  We denote the space of $J$-holomorphic currents from $\alpha$ to $\beta$ by $\mathcal{M}_{\op{current}}(\alpha,\beta)$.  If $J$ is generic, then it is shown in \cite{obg1,obg2} that $d^2 = 0$.  The ECH index, which is the key nonstandard feature of the definition of ECH, will be defined in the next section.  For more about ECH, see \cite{echlecture}.

Now let $\Psi: (X_2,\omega_2) \to \op{int}(X_1,\omega_1)$ be a symplectic embedding of Liouville domains.  Consider the symplectic cobordism
\[ X = (X_1,\omega_1) \setminus \Psi(\op{int}(X_2,\omega_2)).\]
By [HT1], there is an induced map
\[ \Phi: ECH(\partial X_1) \to ECH(\partial X_2).\]
This map is defined by using Seiberg-Witten theory.  Nevertheless, it satisfies a {\em holomorphic curve} axiom.  Namely, it is shown in \cite{cc2} that $\Phi$ is induced from a chain map $\tilde{\Phi}$ with the following property: if $\alpha$ and $\beta$ are nonzero chain complex generators with $\langle \tilde{\Phi}(\alpha),\beta \rangle \ne 0$, then there is a possibly broken {\em $J$-holomorphic current} $C \in \overline{\mathcal{M}}_{\op{current}}(\alpha,\beta)$ with $I(C)=0$.  A {\em broken} $J$-{\em holomorphic current} from $\alpha$ to $\beta$ is a sequence of holomorphic currents $C_1,\ldots,C_n$ such that $C_i \in \mathcal{M}_{\op{current}}(\gamma_i,\gamma_{i+1})$, where the $\gamma_i$ are orbit sets such that $\gamma_1 = \alpha$ and $\gamma_{n+1} = \beta$.  The $C_i$ are called {\em levels}, and in principle could be curves in either $\mathbb{R} \times \partial X_1$ or $\mathbb{R} \times \partial X_2$, with an $\mathbb{R}$ invariant almost-complex structure, or in $\overline{X}$ with a cobordism admissible almost-complex structure. In fact, only one of the levels is a curve in $\overline{X}$; this is called the {\em cobordism} level, and the other levels are called {\em symplectization levels}.    The ECH index of a broken holomorphic current is the sum of the ECH indices of each level.

\subsection{The ECH index and the $J_0$ index}
\label{secn:index}

Let $C \in \mathcal{M}_{\op{current}}(\alpha,\beta)$ be a $J$-holomorphic current in $\overline{X}$.  The ECH index only depends on the relative homology class $[C]$.  Specifically, the formula for the ECH index is as follows:
\begin{equation}
\label{eqn:echindex}
I([C]) = c_{\tau}([C])+Q_{\tau}([C])+CZ_{\tau}^I([C]).
\end{equation}
Here, $\tau$ denotes a symplectic trivialization of $\xi \eqdef \op{Ker}(\mu)$ over each embedded Reeb orbit,  $c_{\tau}([C])$ denotes the relative first Chern class $c_1(T\overline{X}|_{[C]},\tau)$ (defined using an admissible almost-complex structure), $Q_{\tau}([C])$ denotes the ``relative intersection pairing", and $CZ_{\tau}^I([C])$ denotes the {\em total Conley-Zehnder} index
\[ CZ_{\tau}^I([C]) = \sum_i \sum_{l = 1}^{m_i} \op{CZ}_{\tau}(\alpha^l_i) - \sum_{j} \sum_{k=1}^{n_j}\op{CZ}_{\tau}(\beta^k_j),\]
where $\alpha = \sum_i \alpha_i^{m_i}$ and $\beta = \sum_j \beta_j^{n_j} $.  In this formula $\op{CZ}_{\tau}(\gamma^k)$ denotes the Conley-Zehnder index of the $k$-times multiple cover of an embedded Reeb orbit $\gamma$, defined relative to the trivialization $\tau$. We will not define the relative intersection pairing here, see \cite{echlecture} for the details, but in section \ref{secn:ell} we will give formulas for computing these quantities for ellipsoids.

There is a variant of $I$ which bounds the topological complexity of $C$, called the {\em $J_0$} index, which we will also use.  It is given by the formula
\begin{equation}
\label{eqn:joindex}
J_0([C]) \eqdef -c_{\tau}([C])+Q_{\tau}([C])+CZ_{\tau}^J([C]),
\end{equation}
where $CZ_{\tau}^J([C])=\sum_i \sum_{l = 1}^{m_i-1} \op{CZ}_{\tau}(\alpha^l_i) - \sum_{j} \sum_{k=1}^{n_j-1}\op{CZ}_{\tau}(\beta^k_j).$  Assume now that $C$ is somewhere injective, connected, has genus $g$, and all ends at elliptic orbits.  It is shown in \cite[Prop 6.9]{hutchings_absolute} that
\[J_0(C) \ge 2(g-1 + \delta(C)) + \sum_{\gamma} (2n_{\gamma}-1),\]
where the sum is over all embedded Reeb orbits $\gamma$ at which $C$ has ends, $n_{\gamma}$ denotes the total number of ends of $C$ at $\gamma$, and $\delta(C)$ denotes an algebraic count of the number of singularities of $C$; in particular, $\delta(C) \ge 0$, and equal to $0$ if and only if $C$ is embedded.

\subsection{The partition conditions}

Let $C \in \mathcal{M}(\alpha,\beta)$ be a connected somewhere injective curve in $\overline{X}$ with $I(C)=\op{ind}(C)=0$.  It is shown in \cite{hutchings_absolute} that we can compute the multiplicities of the ends of $C$ at $\alpha$ and $\beta$ purely combinatorially, given the monodromy angles of the underlying embedded orbits in $\alpha$ and $\beta$.  This works as follows for the positive ends, in the case where all orbits in $\alpha$ and $\beta$ are elliptic. (The formula for the negative ends is similar, but we will not need this.  The formula when there are hyperbolic orbits is also not hard.)

Suppose then that $C$ is a somewhere injective curve with positive ends at an elliptic orbit $\tilde{\gamma}$, with total multiplicity $m$.  This means that the positive ends of $C$ form an unordered partition $(m_1,\ldots,m_n)$ of $m$, called the {\em positive partition} of $m$.  Let $\gamma$ be the underlying embedded orbit for $\tilde{\gamma}$ (to clarify the notation, this means that $\tilde{\gamma}$ is an $m$-fold cover of $\gamma$).

Here is how we can compute the partition $(m_1,\ldots,m_n)$.  As $\gamma$ is elliptic our trivialization $\tau$ is homotopic to one where the linearized Reeb flow generates a rotation through an angle $2 \pi \theta$. Then $\theta$ is the monodromy angle for $\gamma$, and we let $L$ be the line in the $xy$-plane that goes through the origin and has slope $m \theta$.  Now let $\Lambda$ be the maximum concave piecewise linear continuous lattice path that starts at $(0,0),$ ends at $(m,\lfloor m \theta \rfloor)$, and stays below the line $L$; this means that the area under $\Lambda$ is the convex hull of the set of lattice points in the region bounded by the $x$-axis, the line $x = m$, and the line $L$.  It is shown in \cite{hutchings_absolute} that the entries $m_i$ are the horizontal displacements of the vectors in $\Lambda$.

\subsection{The ellipsoid case}
\label{secn:ell}

We now explain how to compute $I$ and $J_0$ in the case relevant to Proposition~\ref{prop:mainprop}.

Recall the notation from the beginning of this section, and let $C$ be a $J$-holomorphic current in $\overline{X}$ (to emphasize, $\overline{X}$ now denotes the completion of the cobordism induced by the embedding of the ellipsoids at the beginning of this section).  We can trivialize the contact structure over each embedded Reeb orbit on the boundary of either $E_1$ or $E_2$ by using the identification $T \R^4 = \C \oplus \C$ and observing that the contact structure on the boundary of either ellipsoid restricts to each Reeb orbit as one of these $\C$ factors.  Call this trivialization $\tau$.  Now assume that $C$ is asymptotic to the orbit set $\alpha_1^{m_1}\alpha_2^{m_2}$ at $+\infty$, and asymptotic to the orbit set $\beta_1^{n_1}\beta_2^{n_2}$ at $-\infty$.  We now have the following formulas for the quantities that enter into $I$ and $J_0$:
\[ c_{\tau}([C]) = (m_1 + m_2) - (n_1 + n_2), \quad \quad Q_{\tau}([C])=2(m_1m_2-n_1n_2).\]
We also know that the monodromy angle (with respect to $\tau$) of any of the four embedded Reeb orbits relevant to the asymptotics is equal to the length of this Reeb orbit, divided by the length of the other Reeb orbit (so, for example, the monodromy angle of $\alpha_1$ is slightly less than $1$); we can use this to compute $\op{CZ}^I_{\tau}$ or $\op{CZ}^J_{\tau}$.  These formulas are proved in \cite{echlecture}, see [Ex. 1.8 and Sec. 3.7].  In this section of \cite{echlecture}, Hutchings is considering the case of the symplectization of a single ellipsoid; however, since these quantities are purely topological the computations extend to our situation as well.

The following basic consideration will also be useful:

\begin{fact}
\label{fct:ellipsoidfact}
Let $b/a$ be irrational.  Then the chain complex differential $d$ for $ECH(\partial E(a,b))$ satisfies $d=0$.
\end{fact}
\begin{proof}
As explained in section \ref{sec:prpn}, as our ellipsoids are irrational the Reeb vector field on the boundary has exactly two closed orbits, and they are both elliptic.  Fact~\ref{fct:ellipsoidfact} now follows, since it is shown for example in \cite[Lem. 4.1]{weinsteinstable} that the ECH chain complex differential vanishes for any nondegenerate contact manifold with only elliptic orbits.
\end{proof}

\subsection{Proof of the proposition}

We now have all the ingredients needed to prove Proposition~\ref{prop:mainprop}.

{\em Step 1.}  As stated above, the symplectic cobordism $X = E_1 \setminus \Psi(E_2)$ induces a map
\[ \Phi: ECH(\partial E_1) \to ECH(\partial E_2).\]
This map must be an isomorphism.  The reason for this is that the cobordism $X$ is diffeomorphic to a product, and the ECH cobordism map agrees\footnote{Indeed, this is currently the definition of the ECH cobordism map.} with the cobordism map on Seiberg-Witten Floer cohomology, which is known to be an isomorphism for product cobordisms.  Now consider the ECH generator $\alpha_2^{g_{n+1}}$.  By Fact~\ref{fct:ellipsoidfact}, we know that the ECH chain complex differential vanishes for the boundary of any irrational ellipsoid.  Hence, $[\alpha_2^{g_{n+1}}]$ defines a nonzero class in $ECH(\partial E_1)$.  Thus, $\Phi([\alpha_2^{g_{n+1}}]) \ne 0$.  We know by the ``holomorphic curve" axiom that for any orbit set $\Theta$ appearing in $\Phi([\alpha_2^{g_{n+1}}]) \ne 0$, there is a possibly broken $J$-holomorphic current from $\alpha_2^{g_{n+1}}$ to $\Theta$, of total ECH index $0$.

{\em Step 2.}  We will first explain why we must have $\Theta = \beta_1^{g_{n+2}}.$  This will follow from ECH index calculations for ellipsoids, together with the fact that the holomorphic building from $\alpha_2^{g_{n+1}}$ to $\Theta$ has total ECH index $0$.

First note that, as explained in the section \ref{secn:index}, the ECH index of any $J$-holomorphic current in $\overline{X}$ only depends on the asymptotics of the current.  It also follows from the calculations in section \ref{secn:ell} that there is a canonical  $\Z$-grading for ECH of the boundary of any ellipsoid with the property that the ECH cobordism map must preserve this grading.  This grading is given by
\[ \op{gr}(\gamma_1^{x_1}\gamma_2^{x_2}) = x_1 + x_2 + 2x_1x_2 + \op{CZ}_{\tau}^I(\gamma_1^{x_1}\gamma_2^{x_2}),\]
where $\tau$ is the trivialization used in section \ref{secn:ell}.  It turns out that
\begin{equation}
\label{eqn:gradingequation}
\op{gr}(\gamma_1^{x_1}\gamma_2^{x_2}) = 2 \# \lbrace (a,b) | \mathcal{A}( \gamma_1^{a}\gamma_2^{b}) < \mathcal{A}(\gamma_1^{x_1}\gamma_2^{x_2}) \rbrace,
\end{equation}
where $\mathcal{A}$ denotes the symplectic action, and $a$ and $b$ are both nonnegative integers.  This can be proved directly by interpreting $\op{gr}$ as a count of lattice points in a triangle determined by $(x_1,x_2)$, see e.g. Ex. 3.11 in \cite{echlecture}, but also follows from the fact that the ECH ``$U$"-map is a degree $-2$ isomorphism on ECH$(S^3)$ which decreases the action.

It follows from this that if $\epsilon$ is sufficiently small then $\op{gr}(\alpha_2^{g_{n+1}})= g_{n+1}^2+3g_{n+1}$.  It also follows from the arguments in the previous paragraph that there is a unique orbit set in any grading (this also follows from the computation of the ECH of $S^3$).  We now claim that
\[ \# \lbrace (a,b) | \mathcal{A}(\beta_1^{a}\beta_2^{b}) < \mathcal{A}(\beta_1^{g_{n+2}}) \rbrace = \frac{g_{n+1}^2+3g_{n+1}}{2}.\]  To see this, one computes
\begin{align*}
& \# \lbrace (a,b) | \mathcal{A}(\beta_1^{a}\beta_2^{b}) < \mathcal{A}(\beta_1^{g_{n+2}}) \rbrace \quad\quad\\
& =  \# \lbrace (a,b) | a + b (\frac{g_{n+2}}{g_n}+\epsilon) < g_{n+2} \rbrace \quad \quad \\
& =  \sum_{m=0}^{g_{n}-1} ( \lfloor g_{n+2}-m(\frac{g_{n+2}}{g_n}+\epsilon) \rfloor + 1) \quad \quad \\
& =  g_{n}g_{n+2} - \sum_{m=1}^{g_{n}-1} \lfloor m\frac{g_{n+2}}{g_n} \rfloor \\
& =  \frac{g_ng_{n+2}+g_{n+2}+g_{n}-1}{2} \\
& = \frac{g_{n+1}^2+3g_{n+1}}{2},
\end{align*}
where the second to last equality follows from the identity
\[ \sum_{i=0}^{q-1} \lfloor \frac{ip}{q} \rfloor = \frac{(p-1)(q-1)}{2}\]
for relatively prime positive integers $p$ and $q$, and the last line follows from a Fibonacci identity which is easily proved by induction.  Note that one can also prove by induction that $g_n$ and $g_{n+2}$ are always relatively prime.  It now follows that in fact $\Phi([\alpha_2^{g_{n+1}}]) = [\beta_1^{g_{n+2}}]$.

{\em Step 3.}  Because $\Phi([\alpha_2^{g_{n+1}}]) = [\beta_1^{g_{n+2}}]$, it follows from the properties of the ECH cobordism map explained in \S\ref{sec:prelim} that there is a broken $J$-holomorphic current $Z$ from $\alpha_2^{g_{n+1}}$ to $\beta_1^{g_{n+2}}$.  This broken current could in principle consist of multiple levels, and multiple connected components.  Call a symplectization level {\em trivial} if it is a union of branched covers of trivial cylinders.

\begin{claim}
\label{clm:mainclaim}
The current $Z$ has a single nontrivial level, consisting of a single somewhere injective (in fact embedded) connected component.
\end{claim}

Note that such a level is necessarily a cobordism level.  Claim~\ref{clm:mainclaim} will follow from the following:

\begin{lemma} \label{actions}
Fix orbit sets $\alpha$ on $ \partial \lambda E(\frac{g_{n+2}}{g_{n+1}},\frac{g_{n+2}}{g_{n+1}}+\epsilon)$ and $\beta$ on $\partial E(1,\frac{g_{n+2}}{g_{n}}+\epsilon)$.  Let $C \in \mathcal{M}(\alpha,\beta)$ be a somewhere injective connected $J$-holomorphic curve in $\overline{X}$, and asume that $\op{gr}(\alpha) < \op{gr}(\alpha_2^{g_{n+1}})$.  If $\epsilon$ is sufficiently small, and $\lambda$ is sufficiently close to $1$, then $\mathcal{A}(C) > (\lambda-1)g_{n+2} + \epsilon \lambda g_{n+1}$.
\end{lemma}

Here, the {\em action} $\mathcal{A}(C)$ is defined by
\[ \mathcal{A}(C) = \mathcal{A}(\alpha) - \mathcal{A}(\beta).\]

\begin{proof}

Given positive real numbers $a,b$, let $\mathcal{N}(a,b)$ denote the sequence whose $k^{th}$ term (indexed starting at $0$) is the $(k+1)^{st}$ smallest element in the matrix $(ma+nb)_{m,n \in \Z_{\ge 0}}$.  The motivation for studying this sequence is as follows.  By Step $2$, $\mathcal{N}(a,b)_k$ is the action of the unique ECH generator for $\partial E(a,b)$ in grading $2k$, if $a/b$ is irrational.  Moreover, for fixed $k$, $\mathcal{N}(a,b)_k$ is a continuous function of $a$ and $b$ (in fact, $\mathcal{N}(a,b)$ is the sequence of ECH capacities of the ellipsoid $E(a,b)$).

With this in mind, note first that we know from the calculations in Step $2$ that
\[\mathcal{N}(1, g_{n+2}/g_n)_{(g^2_{n+1}+3g_{n+1})/2}  = \mathcal{N}(g_{n+2}/g_{n+1},g_{n+2}/g_{n+1})_{(g^2_{n+1}+3g_{n+1})/2}.\]
Moreover, it follows from \cite{ms}, see also e.g. \cite{acg}, that
\[\mathcal{N}(1, g_{n+2}/g_n)_k \le \mathcal{N}(g_{n+2}/g_{n+1},g_{n+2}/g_{n+1})_k\] for all $k$.  In fact, we now claim that we must have $\mathcal{N}(1, g_{n+2}/g_n)_k < \mathcal{N}(g_{n+2}/g_{n+1},g_{n+2}/g_{n+1})_k$ for $0 \le k < (g^2_{n+1}+3g_{n+1})/2.$  Otherwise, there would exist nonnegative integers $x$, $y$ and $T$ such that
\[ x + \frac{g_{n+2}}{g_n}y = \frac{g_{n+2}}{g_{n+1}}T,\]
or, rearranging,
\[ \frac{g_{n+1}}{g_{n+2}}x + \frac{g_{n+1}}{g_n}y = T,\]
with $T \le g_{n+1}$.  One can show by induction that $g_{n+1}$ and $g_ng_{n+2}$ are relatively prime; hence, we must have $T = g_{n+1}$ and so \[\frac{g_{n+2}}{g_{n+1}}T = \mathcal{N}(g_{n+2}/g_{n+1},g_{n+2}/g_{n+1})_{(g^2_{n+1}+3g_{n+1})/2}. \]
Arguing as in Step $2$, we see that if $k < (g^2_{n+1}+3g_{n+1})/2$, then $\mathcal{N}(1,g_{n+2}/g_n)_k < \mathcal{N}(1,g_{n+2}/g_n)_{(g^2_{n+1}+3g_{n+1})/2},$ and hence the claim follows.

We can now complete the proof of the Lemma.  Since $C$ is somewhere injective, we must have $I(C) \ge 0$, for example by \cite{hutchings_absolute}. Therefore $\op{gr}(\beta) \le \op{gr}(\alpha)$. As explained in Step $2$, the action of a generator of $ECH(\partial E(a,b))$ is a strictly increasing function of its grading.  Thus, $\mathcal{A}(\alpha)-\mathcal{A}(\beta) \ge \mathcal{A}(\alpha)-\mathcal{A}(\beta')$, where $\beta'$ is the unique orbit set with $\op{gr}(\beta')=\op{gr}(\alpha) < \op{gr}(\alpha_2^{g_n+1})$  By the above claim, if $\epsilon$ is sufficiently small and $\lambda$ is close enough to $1$, then by continuity of the functions $\mathcal{N}(a,b)_k$ we have that $\mathcal{A}(\alpha)-\mathcal{A}(\beta')$ is bounded below by some fixed positive number independent of $k=\op{gr}(\alpha)$.  This implies the Lemma by again choosing $\epsilon$ sufficiently small and $\lambda$ close enough to $1$.

\end{proof}

We now explain why the lemma implies Claim~\ref{clm:mainclaim}.  Assume that there was such a building, and look at the cobordism level.  This consists of a (possibly disconnected) holomorphic current $B$, with $\mathcal{A}(B) \le (\lambda-1)g_{n+2} + \epsilon \lambda g_{n+1}$ (the action difference between $\alpha_2^{g_{n+1}}$ and $\beta_1^{g_{n+2}}$).  Look at the underyling somewhere injective curve for any component $\tilde{C}$ of this current; this curve $C$ must also satisfy $\mathcal{A}(C) \le (\lambda-1)g_{n+2} + \epsilon \lambda g_{n+1}$, and if $\tilde{C}$ is an honest multiple cover, then $C$ must be asymptotic at $+\infty$ to an orbit set with action strictly less than $\alpha_2^{g_{n+1}}$, and hence grading strictly less than the grading of $\alpha_2^{g_{n+1}}$.  It follows from Lemma \ref{actions} that there are no such curves if $\epsilon$ is close to $0$ and $\lambda$ close to $1$; hence, $\tilde{C}$ must be somewhere injective.  The same argument shows that $B$ must consist of a single connected component.

It now follows by general properties of the ECH index, see \cite{hutchings_absolute}, that $I(B) \ge 0$.  Since the total index of the building is $0$, and $I(S) \ge 0$ for any symplectization level of the building, with equality if and only if $S$ is a union of branched covers of trivial cylinders (again by general properties of the ECH index), Claim~\ref{clm:mainclaim} now follows.

{\em Step 4.}  We can now complete the proof of Proposition~\ref{prop:mainprop}.  By the previous steps, there is a connected somewhere injective curve  $C \in \mathcal{M}(\alpha_2^{g_{n+1}},\beta_1^{g_{n+2}})$.  It remains to show that this curve has the properties claimed in the proposition.

First, note that the partition conditions from earlier in this section show that if $\epsilon$ is small enough, then $C$ has $g_{n+1}$ positive ends.  This is because one can make the monodromy angle for $\alpha_2$ arbitrarily small and positive mod $1$ by making $\epsilon$ sufficiently small, so that this claim follows by the definition of the positive partition.

We know by the formulas in section \ref{secn:index} that for any current $C$,
\[I([C])-J_0([C]) = 2c_{\tau}([C])+ \op{CZ}_{\tau}^{\op{top}}([C]),\]
where $\op{CZ}_{\tau}^{\op{top}}([C]) = \sum_i \op{CZ}_{\tau}(\alpha_i^{m_i}) - \sum_j \op{CZ}_\tau (\beta_j^{n_j})$.  We also know that for our particular somewhere injective curve $C$ we have $I(C) = 0$.  It follows from this, and the formulas in sections \ref{secn:index} and \ref{secn:ell}, that
\[ J_0(C) = 2g_{n+2} - 4g_{n+1} + 2 (g_n-1).\]
By the inequality on $J_0$ at the end of section \ref{secn:index} we therefore have
\[ 2(g-1+\delta(C)) + \sum_{\gamma}(2n_{\gamma} - 1) \le 2g_{n+2} - 4g_{n+1} + 2 (g_n-1),\]
hence
\[ 2(g+x) + \delta(C) \le 2(g_{n+2}-3g_{n+1}+g_n)+2 = 2,\]
where $x$ denotes the number of negative ends of $C$, and we have applied an identity for odd index Fibonacci numbers that is easily proved by induction.  It follows that $g=\delta(C)=0$, and $x = 1$.  This proves Proposition~\ref{prop:mainprop}.

\end{section}

\begin{section}{Holomorphic curves for the product embedding.}\label{secn3}

We begin by describing our cobordism and then the moduli space of interest; the basic setup we describe here is similar to the setup in \cite[\S 3]{hk}. We want to understand embeddings in any dimension $2N \ge 6$, but for the analysis in this section we will assume that $N=3$. This simplifies the notation, but does not result in any loss of generality because we will not use any index formulas which may be dimension dependent.

Recall from Theorem \ref{cfn} that there is an embedding
\[ \Phi: E(1,b_n+\epsilon) \to \op{int}(E(c,c+\epsilon)),\] where $b_n = \frac{g_{n+2}}{g_n}$ and $c$ can be chosen slightly larger than $g_{n+2}/g_{n+1}.$
Let $X$ be the cobordism associated to this embedding, and let $\overline{X}$ be the manifold obtained by attaching cylindrical ends to $X$.

For any $S$ we can prolong the embedding $\Phi$ to a map
\[ \Psi: E(1,b_n+\epsilon,S) \to \op{int}(E(c,c+\epsilon)) \times \mathbb{C},\]
given by
\[ (z_1,z_2,z_3) \to (\Phi(z_1,z_2),z_3).\]

The projection of the image of $\Psi$ to the $\C$ factor lies inside some large open disc $B^2(T)$.  The map $\Psi$ therefore induces an embedding
\[ \tilde{\Psi}:E(1,b_n+\epsilon,S) \to \op{int}(E(c,c+\epsilon)) \times \mathbb{C}P^1(2T),\]
where $T$ is some large real number that we will say more about later.  It is convenient to think of $B^2(T)$ as embedded in $\mathbb{C}P^1(2T)$ as the lower hemisphere. We can remove the image of $\tilde{\Psi}$ to get a symplectic cobordism $M$.  Attach ends to $M$ to get a completed symplectic manifold $\overline{M}$.

The manifold $\partial E(1,b_n+\epsilon,S)$ is contact, and the manifold $\partial E(c,c+\epsilon) \times \mathbb{C}P^1(2T)$ has a natural stable Hamiltonian structure. We will study $J$-holomorphic cuves asymptotic to appropriate Reeb orbits, for these stable Hamiltonian structures.

To specify these orbits, first note that we can regard the orbit $\beta_1$ from section \ref{secn2} as an orbit on $\partial E(1,b_n+\epsilon,S)$.  This orbit is non-degenerate.  We can regard the orbit $\alpha_2$ from section \ref{secn2} as an orbit on $\partial E(c,c+\epsilon) \times \mathbb{C}P^1(2T),$ by thinking of it as $\alpha_2 \times \lbrace p \rbrace$.  The point $p$ is chosen as follows.  There is an $S^1$ action on $\mathbb{C}P^1(2T)$, with a unique fixed point in the image of the projection of our embedding to the $\mathbb{C}P^1$ factor.  This is the point $p$.    It will be convenient to choose a coordinate $z_3$ on a neighborhood of this fixed point, such that $z_3=0$ is the fixed point.

We now specify the set of almost complex structures that we want to consider.  First note that $E(1,b_n+\epsilon,S)$ and $E(c,c+\epsilon) \times \mathbb{C}P^1(2T)$ both have an $S^1$ action, given by acting on the third factor.  Moreover, we can arrange it so that the embedding $\tilde{\Psi}$ is equivariant with respect to this action.  Thus, the manifold $\overline{M}$ has an $S^1$ action.  Later we will want to choose $J$ to take advantage of this.  Also, we will want to choose $J$ so that the curves we want to study avoid the point at $\infty$ in $\mathbb{C}P^1(2T)$.  To accomplish this, denote by $U(T)$ the subset \[\overline{E(c,c+\epsilon)} \times ((\mathbb{C}P^1(2T) \setminus B^2(T)),\]
where $\overline{E(c,c+\epsilon)}$ denotes the completion formed by attaching a cylindrical end.  Then, since $\tilde{\Psi}E(1,b_n+\epsilon,S)$ does not intersect this subset, we can regard $U(T)$ as a subset of $\overline{M}$.

Now first let $\mathcal{J}(T)$ denote the space of cobordism admissible almost complex structures on $\overline{M}$.  Also, fix a positive real number $R$.  Let $C$ be a curve in $\overline{M}$ asymototic to orbits $\alpha$ at $(\partial E(c,c+\epsilon)) \times \mathbb{C}P^1(2T)$ and $\beta$ at $\partial E(1,b_n+\epsilon,S)$.  We can define the {\em action} of $C$ as before by
\[ \mathcal{A}(C) = \mathcal{A}(\alpha) - \mathcal{A}(\beta).\]
Let $\mathcal{J}_R(T)$ denote the space of almost complex structures such that any curve $C$ with $\mathcal{A}(C) \le  d\pi R^2$ and one negative end has image contained in the interior of $U(T)^c$.  Here, $d=g_{n+1}$.

\begin{lemma}
\label{lem:jexist} Given $R$, for sufficiently large $T$, the space $\mathcal{J}_R(T)$ is open and nonempty.
\end{lemma}
\begin{proof}

This is proved as in \cite[Lem. 3.3]{hk}.

\end{proof}

We will now write $\mathcal{J}_R$ instead of $\mathcal{J}_R(T)$ when the explicit value of $T$ is not needed and all we need to know is that we have chosen $T$ large enough so that Lemma~\ref{lem:jexist} applies.  Now let $\overline{\mathcal{J}}_R \subset \mathcal{J}_R$ denote the space of $S^1$ invariant almost complex structures, and recall the cobordism $\overline{X}$ from \ref{sec:prpn}.  There is an inclusion $\overline{X} \subset \overline{M}$ induced by the map $(z_1,z_2) \to (z_1,z_2,p)$.  Note that if $J \in \overline{\mathcal{J}}_R,$ then this inclusion is $J$-holomorphic.  As in \cite{hk}, say that a $J \in \overline{\mathcal{J}}_R(T)$ is {\em suitably restricted} if its restriction to $\overline{X}$ is regular for all somewhere injective, finite energy curves of genus $0$ in $\overline{X}$.

We can now state the main goals of this section. Given $J \in \overline{\mathcal{J}}_R$, let $\mathcal{M}_{J}(g_{n+1}\alpha_2,\beta_1^{g_{n+2}})$ denote the moduli space of genus $0$ somewhere injective $J$-holomorphic curves in $\overline{M}$ with $g_{n+1}$ positive punctures and one negative puncture, that are asymptotic to translations of $\alpha_2$ at positive infinity, and asymptotic to $\beta_1^{g_{n+2}}$ at negative infinity.  (Note that for curves in higher dimensional cobordisms, we will always specify the number of positive and negative punctures and the corresponding multiplicities, rather than just the total orbit set).  In \S\ref{secn4}, we will show that this space is compact.   For $S^1$ invariant $J$ the moduli space $\mathcal{M}_0$ from the end of section \ref{sec:prpn} of curves in $\overline{X}$ is naturally a subset of $\mathcal{M}_{J}(g_{n+1}\alpha_2,\beta_1^{g_{n+2}})$. It turns out that both of these moduli spaces have virtual dimension $0$; a discussion of the index formulas in higher dimension is postponed until section \ref{secn4}.

The first result is the following.
\begin{proposition}\label{nontriv}
If $J \in \overline{\mathcal{J}}_R$ is regular and suitably restricted, then the oriented cobordism class of $\mathcal{M}_{J}(g_{n+1}\alpha_2,\beta_1^{g_{n+2}})$ is nontrivial.
\end{proposition}
We prove this in section \ref{3pt1}.

We will combine Proposition \ref{nontriv} with the following.
\begin{proposition} \label{jexist} The set of $J \in \overline{\mathcal{J}}_R$ which are regular for $\mathcal{M}_J(g_{n+1}\alpha,\beta_1^{g_{n+2}})$ and suitably restricted is nonempty.
\end{proposition}
This is established in section \ref{3pt2}.

\begin{subsection}{The moduli space for invariant almost-complex structures.}\label{3pt1}

Here we prove Proposition \ref{nontriv} (modulo the compactness of the moduli space, which is deferred to \S\ref{secn4}).

\begin{proof}

The proof is similar to \cite[Prop. 3.15]{hk}.  The key will be a version of automatic transversality in this setting established by Wendl in \cite{wendl}.

{\em Step 1.}{\em Splitting the normal bundle.}

Let $C$ be a curve in $\mathcal{M}_{J}(g_{n+1}\alpha_2,\beta_1^{g_{n+2}})$.  Then $C$ is index $0$ (we will prove this in Lemma~\ref{mainspace}), somewhere injective, and transverse, hence by \cite[Cor. 3.17]{wendl} immersed, since $J$ is regular.  Also, since $C$ is transverse, its projection to $\mathbb{C}P^1$ lives in the fixed point set (otherwise, it would not be rigid), and thus its projection is $p$.  Hence $C \in \mathcal{M}_0$. Now let $N$ denote the normal bundle to $C$.  A linear Cauchy-Riemann type operator is a map:
\[ D: \Gamma(N) \to \Gamma(T^{0,1} C \otimes N).\]

We first claim that the bundle $N$ splits as a sum of complex line bundles.
\[ N = H \oplus V.\]
Here, $H$ and $V$ are defined as follows.  First, note that $C$ is a symplectic submanifold of $\overline{M}$.  We can therefore identify its normal bundle with a subbundle of $T\overline{M}|_C$.  Now, a point $y$ on $C$ either maps to the complement of the image of $E(1,b_n+\epsilon,S)$, or to the cylindrical end $\partial E(1,b_n+\epsilon,S) \times (-\infty,0].$  In the first case, we can write $T\overline{M}|_y = T_{\pi(y)}(\overline{X}) \times T_p(\mathbb{C}P^1)$, and in the second case we can write $T\overline{M}_y = \mathbb{C}^3$; here $\pi$ denotes the canonical projection, and we are thinking of $\partial E(1,b_n+\epsilon,S) \times (-\infty,0]$ as identified with the complement of the origin in $E(1,b_n+\epsilon,S)$.  We define $V$ to be the subbundle that is parallel to the $T_p(\mathbb{C}P^1)$ factor in the first case, and the $\lbrace z_3 \rbrace$ factor in the second, and we define $H$ to be the subbundle that is parallel to the $T_{\pi(y)}(\overline{X})$ factor in the first case, and the $\lbrace z_1,z_2 \rbrace$ factors in the second.  Note that this is well-defined.

The argument in \cite[Lem. 3.17]{hk} now says that this is in fact a $J$-holomorphic splitting of complex subbundles.  In that argument, the map $\phi_0$ is induced from an inclusion, but it generalizes to this case without change: all we need is that the map is induced by an embedding which restricts to the $3^{rd}$ coordinate as the identity.

{\em Step 2.}{\em Counting with sign.}

By Proposition \ref{prop:mainprop}, there is at least one element $C$ in the moduli space  $\mathcal{M}_0 \subset \mathcal{M}_{J}(g_{n+1}\alpha_2,\beta_1^{g_{n+2}})$.  If there are no other curves then the Proposition follows immediately from regularity of $J$.  Assume then that there is some other curve $C'$.  We claim that $C'$ counts with the same sign as $C$.  To compute the difference in sign between $C$ and $C'$, we identify their normal bundles, and choose a family of linear Cauchy-Riemann type operators interpolating between their deformation operators, with the same asymptotics.  As in \cite[\S 3.3.1]{hk}, the difference in sign is then given by computing a sum of crossing numbers; these crossing numbers are computed at parameter values where the relevant Cauchy-Riemann type operator has a nontrivial cokernel.

Now note that by Step $1$, the deformaton operator for either $C$ or $C'$ splits as a sum with respect to this splitting:
\[ D=\left( \begin{array}{ccc}
D_H & 0\\
0 & D_V \\
\end{array} \right).\]
We can choose our interpolating family $D(t)$ to respect this splitting.  Now let $D(t_0)$ be some operator in this family.  Then we claim that the cokernel of $D(t_0)$ is trivial.  This is because the normal Chern numbers of the operators $D(t_0)_H$ and $D(t_0)_V$ are negative, so we can appeal to \cite[Thm. 1.2]{wendl}.  To see why they are negative, note first that we can identify the bundle $H$ with the normal bundle to $C$ in $X$, and by the conditions on $J$, we can choose this identification such that the operator $D(t_0)_H$ has the same asymptotics as the deformation operator for $C$ in $X$.

As for the operator $D(t_0)_V$, note first that the stable Hamiltonian structure on $\partial E(c,c+\epsilon) \times \mathbb{C}P^1$ restricts to a stable Hamiltonian structure on $\alpha_2 \times  \mathbb{C}P^1$.  The operator $D(t_0)_V$ is asymptotic at any positive puncture $q_i$ at a Reeb orbit $\alpha_2 \times \lbrace p \rbrace$ to the asymptotic operator on $T_p(\mathbb{C}P^1)|_{\alpha_2 \times \lbrace p \rbrace}$ induced by the Reeb flow for this stable Hamiltonian structure.  We claim that the orbit at $q_i$ has odd parity, in the sense of \cite[Sec. 3.2]{wendl}; given this, the claimed fact about the normal Chern number will follow.  To see why the parity is odd, note that we can choose a trivialization for $T_p(\mathbb{C}P^1)$ such that this asymptotic operator is $i\partial_t$, where $t$ is the angular coordinate near the puncture, so if we perturb this operator by adding a constant, the perturbed operator will have odd Conley-Zehnder index.  By the definitions in \cite[Sec. 3.2]{wendl}, this says that these orbits count as odd in their contribution to the normal Chern number.  Thus, the operator $D(t_0)_V$ cannot have a nontrivial cokernel either, so we are done.

\end{proof}
\end{subsection}

\begin{subsection}{Regular and invariant structures exist.}\label{3pt2}

Here we prove Proposition \ref{jexist}, that is, we establish the existence of suitably restricted almost-complex structures $J \in \overline{\mathcal{J}}_R$ that are regular for curves in $\mathcal{M}_J(g_{n+1}\alpha_2,\beta_1^{g_{n+2}}).$

We follow the methods of \cite{hk} closely, and the first observation is that standard transversality arguments imply the existence of suitably restricted almost-complex structures $J \in \overline{\mathcal{J}}_R$ which are regular for curves in $\mathcal{M}_{J}(g_{n+1}\alpha_2,\beta_1^{g_{n+2}})$ which are {\em orbitally simple}, that is, curves which intersect at least one orbit of the $S^1$ action exactly once and transversally, see \cite{hk}, section $3.3.2$.

We may suppose that our embedding $\Phi$ extends to a slightly larger ellipsoid $(1+\delta)E(1,b_n+\epsilon)$. Denote by $$\Sigma = \Phi(\partial (1+\delta)E(1,b_n+\epsilon)) \times \C P^1(2T).$$  This is a stable Hamiltonian hypersurface in $\overline{M}$. Furthermore, the same transversality arguments allow us to find suitably restricted $J^K \in \overline{\mathcal{J}}_R$ which are regular for orbitally simple curves and also satisfy the following conditions:

$\bullet$ The almost-complex structure $J^K$ is stretched to length $K$ along $\Sigma$;

$\bullet$ Away from $\Sigma$ the $J^K$ converge smoothly to a fixed almost-complex structure;

$\bullet$ On $(\overline{E(c,c+\epsilon)} \setminus \Phi((1+\delta)E(1,b_n+\epsilon))) \times \mathbb{C}P^1(2T)$ the natural projection $\pi:\overline{M} \to \overline{X}$ is $J^K$-holomorphic.

Proposition \ref{jexist} now follows from the next proposition.

\begin{proposition} For $K$ sufficiently large, all curves in $\mathcal{M}_{J^K}(g_{n+1}\alpha_2,\beta_1^{g_{n+2}})$ are orbitally simple.
\end{proposition}

\begin{proof}
The proof is analogous to \cite{hk}, Proposition $3.18$. We argue by contradiction and suppose $u_K$ is a curve in $\mathcal{M}_{J^K}(g_{n+1}\alpha_2,\beta_1^{g_{n+2}})$ which is not orbitally simple.

Taking a limit as $K \to \infty$ the compactness theorem of Symplectic Field Theory \cite{behwz} implies that a subsequence of the $u_K$ converge to a holomorphic building with components in the completion $\overline{A}$ of $\Phi((1+\delta)E(1,b_n+\epsilon) \times \mathbb{C}P^1(2T) \setminus \tilde{\Psi}(E(1,b_n+\epsilon,S))$, the completion $\overline{B}$ of $(E(c,c+\epsilon) \setminus \Phi((1+\delta)E(1,b_n+\epsilon))) \times \mathbb{C} P^1(2T)$ and possibly the symplectizations of $\partial \tilde{\Psi}(E(1,b_n+\epsilon,S))$ and $\Sigma$ and $\partial E(c,c+\epsilon) \times \mathbb{C}P^1(2T)$.

Let $S_K = u_K^{-1}(\overline{B})$. By definition $\pi \circ u_K|_{S_K}$ is a nontrivial multiple covering onto its image. The degree is constant on each component and by the asymptotic behaviour of the $u_K$ near their positive punctures we see that the degree is bounded by $g_{n+1}$. The convergence implies that for the limiting curves $v$ mapping to $\overline{B}$ the projection $\pi \circ v$ is also a nontrivial multiple cover.

Suppose that $\pi \circ v$ is a multiple cover of a finite energy curve $w$ mapping to the completion of $E(c,c+\epsilon) \setminus \Phi((1+\delta)E(1,b_n+\epsilon))$. Counting with multiplicity suppose that $w$ has $k$ positive ends, $l$ negative ends asymptotic to $\beta_1$ and $m$ negative ends asymptotic to multiples of $\beta_2$. Then up to terms of order $\epsilon, \delta$ the curve $w$ has symplectic area $$k\frac{g_{n+2}}{g_{n+1}} - l - m\frac{g_{n+2}}{g_n}.$$ As the curves $u_K$ have action of order $\epsilon$, so does $w$ and therefore the expression above is $0$. Hence $$\frac{g_{n+2}}{g_n g_{n+1}}(kg_n - mg_{n+1}) \in \Z.$$ Now, consecutive odd index Fibonacci numbers are coprime. Therefore $g_{n+1} | (kg_n - mg_{n+1})$ and $k$ is a multiple of $g_{n+1}$. But as $v$ is a limit of the $u_K$, by area considerations it can have at most $g_{n+1}$ positive ends, and if it does have $g_{n+1}$ positive ends they must be simply covered. As we have seen that $w$ must also have $g_{n+1}$ positive ends this contradicts our assumption that $v$ covers $w$ nontrivially and completes the proof.

\end{proof}

\end{subsection}

\end{section}

\begin{section}{Compactness}\label{secn4}

Continue to consider the manifolds $\overline{M}$ and $\overline{X}$ from the previous section, only we now allow any $N \ge 3$ as a parameter in their construction, rather than restricting to the case $N=3$ as we did in that section. Hence now $\overline{M}$ is a completion of $E(c,c+\epsilon) \times \mathbb{C}P^1(2T)^{N-2}$ with the image of an embedding $\tilde{\Psi}$ removed, where  \[ \tilde{\Psi}:E(1,b_n+\epsilon,S, \dots ,S) \to \op{int}(E(c,c+\epsilon)) \times \mathbb{C}P^1(2T)^{N-2}.\] Similarly to section \ref{secn3} we choose coordinates on $\mathbb{C}P^1(2T)^{N-2}$ such that the fixed point of the $(S^1)^{N-2}$ torus action is $z_3 = \dots = z_N =0$. By choosing $T$ sufficiently large, by Lemma \ref{lem:jexist} we need only consider curves whose projection onto $\mathbb{C}P^1(2T)^{N-2}$ lies in the affine part $\mathbb{C}^{N-2}$.

Let $\alpha_i$ and $\beta_i$ be the embedded closed Reeb orbits on $\partial E(c, c+\epsilon)$ and $\partial E(1,b_n + \epsilon)$ as in the previous sections, and continue to denote an $r$-fold cover of a simple Reeb orbit $\gamma$ by $\gamma^r$.
We denote by $\alpha_{i,w}$ the Reeb orbit $\alpha_i \times \{w\} \subset \partial E(c,c+\epsilon) \times \mathbb{C}P^1(2T)^{N-2}$ and $\beta_i$ will also denote a Reeb orbit in $\partial E(1,b_n + \epsilon,S, \dots ,S)$ using the inclusion $E(1,b_n + \epsilon) \subset E(1,b_n + \epsilon,S, \dots ,S)$.
Fix $J$, an almost complex structure which is compatible with the symplectic form, and having cylindrical ends.

Define $${\cal M}(J) = {\cal M}(\alpha_1^{r_1}, \dots ,\alpha_1^{r_{n_1}}, \alpha_2^{s_1}, \dots ,\alpha_2^{s_{n_2}}; \beta_1^{t_1}, \dots ,\beta_1^{t_{n_3}}, \beta_2^{u_1}, \dots ,\beta_2^{u_{n_4}};J)$$ to be a certain moduli space of $J$-holomorphic spheres in $\overline{M}$ with $n_1 + n_2 + n_3 + n_4$ punctures, quotiented by reparameterizations of the domain. Specifically, require curves $u \in {\cal M}(J)$ to have $n_1$ positive punctures asymptotic to covers of some $\alpha_{1,w}$, with the $i^{th}$ one covering the simple orbit $r_i$ times. Similarly there must be $n_2$ positive punctures asymptotic to covers of the $\alpha_{2,w}$ and so on.

The goal of this section is to show that for the relevant values of the $r_i, s_i, t_i$ and $u_i$, the moduli space ${\cal M}(J)$ is sequentially compact, as is a related moduli space associated to $1$-parameter families of almost-complex structures, see Theorem \ref{compact}. To do this, we first need formulas for the virtual index of holomorphic curves in various cobordisms.  The index formula for holomorphic curves in symplectic cobordisms can be found for example in \cite{egh}; the formulas in the case of ellipsoids were worked out in \cite{hk}.

\begin{proposition}\label{middleindex} For $S$ sufficiently large and $\epsilon$ small, the virtual deformation index of curves $u \in {\cal M}(J)$ is given by
\begin{eqnarray*}
\mathrm{index}(u)  =  (N-3)(2-n_1 - n_2 - n_3 - n_4) \\
 +  \sum_{i=1}^{n_1}(2r_i + 2\lfloor \frac{r_ic}{c+\epsilon} \rfloor + N-1)
 +  \sum_{i=1}^{n_2}(2s_i + 2\lfloor \frac{s_i(c+\epsilon)}{c} \rfloor + N-1) \\
 -  \sum_{i=1}^{n_3}(2t_i + 2\lfloor \frac{t_i}{b_n+\epsilon} \rfloor + N-1)
 - \sum_{i=1}^{n_4}(2u_i + 2\lfloor
 u_i(b_n+\epsilon) \rfloor + N-1) \\
 =  2(N-3) +2n_2 -(2N-4)n_3 - (2N-4)n_4 + 4\sum_{i=1}^{n_1}r_i  + 4\sum_{i=1}^{n_2}s_i  \\
- 2\sum_{i=1}^{n_3}(t_i + \lfloor \frac{t_i}{b_n+\epsilon} \rfloor) - 2\sum_{i=1}^{n_4}(u_i + \lfloor u_i(b_n+\epsilon) \rfloor)
\end{eqnarray*}
\end{proposition}

We note that this index is always even.  Here is an immediate application of Proposition~\ref{middleindex} that we will need:

\begin{lemma}\label{mainspace} Let ${\cal M}(J) = {\cal M}(\alpha_2, \dots ,\alpha_2, \beta_1^{g_{n+2}}; J)$ with $g_{n+1}$ copies of $\alpha_2$, and let $u \in \mathcal{M}(J)$.  Then $\mathrm{index}(u)=0$.
\end{lemma}

\begin{proof} By Proposition~\ref{middleindex}, we have
\begin{eqnarray*}
\mathrm{index}(u) = 2(N-3) + 2g_{n+1} -(2N-4) + 4g_{n+1}
- 2(g_{n+2} + \lfloor \frac{g_{n+2}}{b_n+\epsilon} \rfloor) \\
= -2 + 2g_{n+1} + 4g_{n+1} - 2(g_{n+2} + g_n-1) \\
= 2(3g_{n+1} - g_{n+2} - g_n) =0.
\end{eqnarray*}
In the last line, we have used a standard Fibonacci identity.

\end{proof}

There are similar moduli spaces of curves in the cylindrical manifolds $\partial E(c, c+\epsilon) \times (\mathbb{C}P^1)^{(N-2)}\times \R$ and $\partial E(1,b_n+\epsilon, S, \dots ,S) \times \R$ that we will want to study, where the almost complex structure is assumed $\mathbb{R}$-invariant.

In the first case we study moduli spaces $${\cal M}_{\op{ball}}(J) = {\cal M}(\alpha_1^{r_1}, \dots ,\alpha_1^{r_{n_1}}, \alpha_2^{s_1}, \dots ,\alpha_2^{s_{n_2}}; \alpha_1^{t_1}, \dots ,\alpha_1^{t_{n_3}}, \alpha_2^{u_1}, \dots ,\alpha_2^{u_{n_4}};J)$$ of curves in $\partial E(c, c+\epsilon) \times (\mathbb{C}P^1)^{(N-2)}\times \R$.  Note that as before, we are only requiring the ends lie on the Morse-Bott families corresponding to the $\alpha_i$.  The analogue of Proposition \ref{middleindex} in this case is the following.

\begin{proposition}\label{topindex} The virtual deformation index of curves $u \in {\cal M}_{\op{ball}}(J)$ is given by
\begin{eqnarray*}
\label{top}
\mathrm{index}(u) = (N-3)(2-n_1 - n_2 - n_3 - n_4) \\ + \sum_{i=1}^{n_1}(2r_i + 2\lfloor \frac{r_ic}{c+\epsilon} \rfloor + N-1) + \sum_{i=1}^{n_2}(2s_i + 2\lfloor \frac{s_i(c+\epsilon)}{c} \rfloor + N-1) \\
- \sum_{i=1}^{n_3}(2t_i + 2\lfloor \frac{t_ic}{c+\epsilon} \rfloor - N+3) - \sum_{i=1}^{n_4}(2u_i + 2\lfloor \frac{u_i(c+\epsilon)}{c} \rfloor -N+3) \\
= 2(N-3) +2n_2 +2n_3 + 4\sum_{i=1}^{n_1}r_i  + 4\sum_{i=1}^{n_2}s_i  \\
- 4\sum_{i=1}^{n_3}t_i  - 4\sum_{i=1}^{n_4}u_i.
\end{eqnarray*}
\end{proposition}

The following is an important application of Proposition~\ref{topindex}.

\begin{lemma}\label{topid} Let $u \in {\cal M}_{\op{ball}}(\alpha_2, \dots ,\alpha_2; \gamma; J)$; that is, let $u$ be a curve with positive ends simply covering $\alpha_2$ and a single negative end $\gamma$, which may be a cover.  Then $\mathrm{index}(u) \ge 2(N-2)+2(c-1)$, where $c$ is the covering degree of the end in the case when it covers $\alpha_2$ and is $0$ otherwise. Moreover, there is equality if and only if $u$ covers a cylinder over $\gamma = \alpha_2$.
\end{lemma}

\begin{proof} First we suppose $\gamma = \alpha_1^r$ and $u$ has $k$ positive ends. Then by Proposition~\ref{topindex}
\begin{eqnarray*}
\mathrm{index}(u) = 2(N-3) +2k +2 +  4k  - 4r \\
=2(N-2)+6k-4r.
\end{eqnarray*}
But by area considerations, we may assume $k \ge r$ and so the index is strictly greater than $2(N-2)$.

Now we suppose $\gamma = \alpha_2^r$ and still $u$ has $k$ positive ends. Then
\begin{eqnarray*}
\mathrm{index}(u) = 2(N-3) +2k +  4k  - 4r \\
=2(N-2)+6k-4r-2.
\end{eqnarray*}
Again as $k \ge r$ the index is at least $2(N-2)+2(r-1)$, but now we have equality only if $k=r$, which implies that $u$ covers a cylinder.
\end{proof}

We can do a similar analysis for curves in $\partial E(1,b_n+\epsilon, S, \dots ,S) \times \R$. The relevant moduli spaces are now denoted $${\cal M}_{\op{ellip}}(J) = {\cal M}(\beta_1^{r_1}, \dots ,\beta_1^{r_{n_1}}, \beta_2^{s_1}, \dots ,\beta_2^{s_{n_2}}; \beta_1^{t_1}, \dots ,\beta_1^{t_{n_3}}, \beta_2^{u_1}, \dots ,\beta_2^{u_{n_4}};J),$$
and the corresponding index formula is as follows:

\begin{proposition}\label{bottomindex} The virtual deformation index of curves $u \in {\cal M}_{\op{ellip}}(J)$ is given by
\begin{eqnarray*}
\mathrm{index}(u) = (N-3)(2-n_1 - n_2 - n_3 - n_4) \\
 + \sum_{i=1}^{n_1}(2r_i + 2\lfloor \frac{r_i}{b_n+\epsilon} \rfloor + N-1) + \sum_{i=1}^{n_2}(2s_i + 2\lfloor s_i(b_n+\epsilon) \rfloor + N-1) \\
- \sum_{i=1}^{n_3}(2t_i + 2\lfloor \frac{t_i}{b_n+\epsilon} \rfloor +N-1) - \sum_{i=1}^{n_4}(2u_i + 2\lfloor u_i(b_n+\epsilon) \rfloor +N-1) \\
= 2(N-3) +2n_1+2n_2 -(2N-4)n_3-(2N-4)n_4 \\ + 2\sum_{i=1}^{n_1}(r_i + \lfloor \frac{r_i}{b_n+\epsilon} \rfloor) + 2\sum_{i=1}^{n_2}(s_i + \lfloor s_i(b_n+\epsilon) \rfloor) \\
- 2\sum_{i=1}^{n_3}(t_i + \lfloor \frac{t_i}{b_n+\epsilon} \rfloor) - 2\sum_{i=1}^{n_4}(u_i + \lfloor u_i(b_n+\epsilon) \rfloor).
\end{eqnarray*}
\end{proposition}

Here is an important application of this that we will need:

\begin{lemma}\label{bottomid} Let $u \in {\cal M}_{\op{ellip}}(\beta_1^{r_1}, \dots ,\beta_1^{r_{n_1}}, \beta_2^{s_1}, \dots ,\beta_2^{s_{n_2}}; \beta_1^{g_{n+2}}; J)$; that is $u$ has arbitrary positive ends, but has only a single negative end covering $\beta_1$ $g_{n+2}$ times. Then $\mathrm{index}(u) \ge 0$, with equality if and only if $u$ also has a single positive end covering $\beta_1^{g_{n+2}}$.
\end{lemma}

\begin{proof}

By Proposition~\ref{bottomindex}, we have
\begin{eqnarray*}
\mathrm{index}(u) = 2(N-3) +2n_1+2n_2 -(2N-4) \\
 + 2\sum_{i=1}^{n_1}(r_i + \lfloor \frac{r_i}{b_n+\epsilon} \rfloor) + 2\sum_{i=1}^{n_2}(s_i + \lfloor s_i(b_n+\epsilon) \rfloor) \\
- 2(g_{n+2} + \lfloor \frac{g_{n+2}}{b_n+\epsilon} \rfloor) \\
= 2n_1+2n_2 + 2\sum_{i=1}^{n_1}(r_i + \lfloor \frac{r_i}{b_n+\epsilon} \rfloor) + 2\sum_{i=1}^{n_2}(s_i + \lfloor s_i(b_n+\epsilon) \rfloor)
- 2(g_{n+2} + g_n) \\
\ge 2(\sum_{i=1}^{n_1}(r_i + \frac{r_i}{b_n}) + \sum_{i=1}^{n_2}(s_i + s_i b_n+1)
- g_{n+2} - g_n)
\end{eqnarray*}
with equality here if and only if $\lfloor \frac{r_i}{b_n+\epsilon} \rfloor = \frac{r_i}{b_n} -1$ and $\lfloor s_i(b_n+\epsilon) \rfloor = s_i b_n$ for all $i$. These conditions hold if and only if $\frac{r_i}{b_n}$ and $s_i b_n$ are always integers.

Now, the area inequality for holomorphic curves implies that $\sum_{i=1}^{n_1}r_i + \sum_{i=1}^{n_2}s_i(b_n+\epsilon) \ge g_{n+2}$ and hence for a small choice of $\epsilon$ we have $\sum_{i=1}^{n_1}r_i + \sum_{i=1}^{n_2}s_i b_n \ge g_{n+2}$ and $\sum_{i=1}^{n_1}\frac{r_i}{b_n} + \sum_{i=1}^{n_2}s_i \ge \frac{g_{n+2}}{b_n} = g_n$. It follows that $\mathrm{index}(u) \ge 0$ with equality only if $n_2=0$ and $\sum r_i =g_{n+2}$.

We claim that in the case of equality each $r_i \ge g_{n+2}$. As $\sum r_i =g_{n+2}$ this immediately implies that there is a single positive end and completes the proof of the lemma.

To justify the claim, to have equality we have seen that each $r_i$ must be a multiple of $b_n = \frac{g_{n+2}}{g_n}$, so if the claim were false and some $r_i < g_{n+2}$ then $g_{n+2}$ and $g_n$ have a common factor. Using the identity $3g_{n+1}=g_{n+2} + g_n$ we see that either this common factor is $3$, or all $g_n$ share a prime factor, which is certainly not the case.  However, in fact none of the $g_n$ are divisible by $3$.  This is implied, for example, by the Fibonacci identity $g_n^2+g_{n+1}^2-3g_ng_{n+1}=-1$ (which is shown in \cite{ms}), since $-1$ is not a square mod $3$.

\end{proof}

Now we choose a generic family $\{J_t\}$ of admissible almost-complex structures on $X$, all equal outside of a compact set, and study the universal moduli space ${\cal M} = \{([u],t)| [u] \in {\cal M}^s(J_t), t \in [0,1]\}$ where ${\cal M}^s(J) \subset {\cal M}(\alpha_2, \dots ,\alpha_2, \beta_1^{g_{n+2}}; J),$ consists of somewhere injective curves, with the notation  as in Lemma \ref{mainspace}. The main result of this section is the following.

\begin{theorem}\label{compact} ${\cal M}$ is compact.
\end{theorem}

\begin{proof}

{\em Step 1.} {\em Gathering together curves into components of the holomorphic building.} By the compactness theorem in \cite{behwz}, the limit of curves in ${\cal M}$ is a holomorphic building consisting of curves in $\overline{M}$ and perhaps multiple levels of curves in $\partial E(c, c+\epsilon) \times (\mathbb{C}P^1)^{N-2} \times \R$ and $\partial E(1,b_n+\epsilon, S, \dots ,S) \times \R$ with matching asymptotic limits.

For the purposes of our index calculations, it will be convenient to think of certain subsets of curves with matching ends as glued together to form a single component. This is done as follows:

\begin{enumerate}
\item Any two curves which both lie in levels of $\partial E(c, c+\epsilon) \times (\mathbb{C}P^1)^{N-2} \times \R$ or both lie in levels of $\partial E(1,b_n+\epsilon, S, \dots ,S) \times \R$ and have a matching end are glued together to lie in the same component.

\item Any component without negative ends will be glued with the higher level curves which match its positive ends, and the resulting component will be thought of as a component in the higher level.

\end{enumerate}

To help avoid confusion we will always denote these components with upper case letters and individual curves by lower case letters. Note that Lemmas \ref{topid} and \ref{bottomid} apply also to components mapping to $\partial E(c, c+\epsilon) \times (\mathbb{C}P^1)^{N-2} \times \R$ and $\partial E(1,b_n+\epsilon, S, \dots ,S) \times \R$ defined as above.

After these identifications we will end up with components mapping to $\overline{M}$, each with a single negative end, a single component (perhaps trivial) mapping to $\partial E(1,b_n+\epsilon, S, \dots,S) \times \R$ with a single negative end asymptotic to $\beta_1^{g_{n+2}}$, and perhaps a union of components mapping to $\partial E(c, c+\epsilon) \times (\mathbb{C}P^1)^{N-2} \times \R$. Each of the components in $\partial E(c, c+\epsilon) \times (\mathbb{C}P^1)^{N-2} \times \R$ has positive ends asymptotic to $\alpha_2$ and a single negative end. The control on the negative ends follows because we are taking limits of curves of genus $0$.

Note, however, that it is certainly possible that {\em curves} in $\overline{M}$ (and $\partial E(c, c+\epsilon) \times (\mathbb{C}P^1)^{N-2} \times \R$) have multiple negative ends.

{\em Step 2.} {\em Index estimates.}  We will obtain a useful estimate for the index of curves in $\overline{M}$, and as a result for the index of components in $\overline{M}$.

Suppose that a limiting curve $u$ in $\overline{M}$ lies in a moduli space  $${\cal M}(\alpha_1^{r_1}, \dots ,\alpha_1^{r_{n_1}}, \alpha_2^{s_1}, \dots ,\alpha_2^{s_{n_2}}; \beta_1^{t_1}, \dots ,\beta_1^{t_{n_3}}, \beta_2^{u_1}, \dots ,\beta_2^{u_{n_4}};J).$$

For generic $1$-parameter families of almost-complex structures we may assume that somewhere injective curves in $\overline{M}$ have $\mathrm{index}(u) \ge -1$. Then since all indices are automatically even we have that in fact the index is nonnegative.

In general, suppose that a curve $u$ is a $k$-times cover of a somewhere injective curve $\tilde{u}$. Suppose this curve lies in $${\cal M}(\alpha_1^{\tilde{r}_1}, \dots ,\alpha_1^{\tilde{r}_{\tilde{n}_1}}, \alpha_2^{\tilde{s}_1}, \dots ,\alpha_2^{\tilde{s}_{\tilde{n}_2}}; \beta_1^{\tilde{t}_1}, \dots ,\beta_1^{\tilde{t}_{\tilde{n}_3}}, \beta_2^{\tilde{u}_1}, \dots ,\beta_2^{\tilde{u}_{\tilde{n}_4}};J).$$ This means that the positive ends of $u$ asymptotic to multiples of $\alpha_1$ can be partitioned into $\tilde{n}_1$ blocks according to which end of $\tilde{u}$ they cover. Thus the sum of the $r_i$ in the first block add to $k\tilde{r}_1$ and so on, and similarly for the other limiting orbits.

Proposition \ref{middleindex} gives us the index of $\tilde{u}$ as follows:
\begin{eqnarray}
\label{eqn:underlyingcalculation}
\mathrm{index}(\tilde{u}) = (N-3)(2-\tilde{n}_1 - \tilde{n}_2 - \tilde{n}_3 - \tilde{n}_4) \\
 + \sum_{i=1}^{\tilde{n}_1}(2\tilde{r}_i + 2\lfloor \frac{\tilde{r}_ic}{c+\epsilon} \rfloor + N-1) + \sum_{i=1}^{\tilde{n}_2}(2\tilde{s}_i + 2\lfloor \frac{\tilde{s}_i(c+\epsilon)}{c} \rfloor + N-1) \nonumber \\
- \sum_{i=1}^{\tilde{n}_3}(2\tilde{t}_i + 2\lfloor \frac{\tilde{t}_i}{b_n+\epsilon} \rfloor + N-1) - \sum_{i=1}^{\tilde{n}_4}(2\tilde{u}_i + 2\lfloor \tilde{u}_i(b_n+\epsilon) \rfloor + N-1) \nonumber \\
=2(N-3)+2\tilde{n}_2 - (2N-4)\tilde{n}_3 - (2N-4)\tilde{n}_4 + 4\sum_{i=1}^{\tilde{n}_1}\tilde{r}_i + 4\sum_{i=1}^{\tilde{n}_2}\tilde{s}_i \nonumber \\
- 2\sum_{i=1}^{\tilde{n}_3}(\tilde{t}_i + \lfloor \frac{\tilde{t}_i}{b_n+\epsilon} \rfloor) - 2\sum_{i=1}^{\tilde{n}_4}(\tilde{u}_i + \lfloor \tilde{u}_i(b_n+\epsilon) \rfloor) \nonumber
\end{eqnarray}

By combining Proposition \ref{middleindex} and \eqref{eqn:underlyingcalculation}, and using $\sum r_i = k \sum \tilde{r}_i$, together with similar formulas, we get:
\begin{eqnarray}
\label{eqn:indexestimate}
\mathrm{index}(u) = k \mathrm{index}(\tilde{u}) \hspace{35 mm}\\
+2(1-k)(N-3)+2(n_2-k\tilde{n}_2) + 2(N-2)(k\tilde{n}_3-n_3) +2(N-2)(k\tilde{n}_4-n_4) \nonumber \\
- 2 \sum_{i=1}^{n_3} \lfloor \frac{t_i}{b_n+\epsilon} \rfloor + 2k\sum_{i=1}^{\tilde{n}_3}\lfloor \frac{\tilde{t}_i}{b_n+\epsilon} \rfloor  -2\sum_{i=1}^{n_4}\lfloor u_i(b_n+\epsilon) \rfloor + 2k\sum_{i=1}^{\tilde{n}_4}\lfloor \tilde{u}_i(b_n+\epsilon) \rfloor. \nonumber
\end{eqnarray}

With this formula in hand for curves in $\overline{M}$ we proceed to consider components in $\overline{M}$.

Let us assume that a component $C$ in $\overline{M}$ consists of curves $u^p$ in $\overline{M}$ for $1 \le p \le P$ and components $W^q$ in $\partial E(1,b_n+\epsilon, S, \dots,S) \times \R$ for $1 \le q \le Q$. We assume that the negative end of the component is the negative end of $u^1$ asymptotic to $\beta_1^{t_1}$. The assumption here is that this negative end is asymptotic to a cover of $\beta_1$; the case when it is asymptotic to a cover of $\beta_2$ follows by the same argument.

We denote the numbers of ends and covering numbers of curves $u^p$ using the same notation as above but with a superscript $p$. We define $N_3 = \sum_p n^p_3$ and $N_4 = \sum_p n^p_4$. The total number of matching ends with components $W^q$ is then $N_3 + N_4 -1$ (because one end is unmatched), and each negative end of a $u^p$ (except the first end of $u^1$) matches with a positive end of one of the $W^q$. Finally, as the component has genus $0$ we must have $P+Q = N_3 + N_4$.

As above, $u^p$ will be a $k^p$ times cover of a somewhere injective curve $\tilde{u}^p$, and we use the natural notation to describe the $\tilde{u}^p$.

Our key index estimate can now be stated as follows.

\begin{lemma} \label{keyest} $\mathrm{index}(C) \ge 2\sum_p(n^p_2-k^p\tilde{n}^p_2)$ with equality if and only if the component $C$ contains no curves $W^q$.
\end{lemma}

\begin{proof}
We sum over all curves to get the total index of our component. For curves in $\overline{M}$ we use formula \eqref{eqn:indexestimate} and the fact that somewhere injective curves have nonnegative index. The index of components $W^q$ in $\partial E(1,b_n+\epsilon, S, \dots,S) \times \R$ with no negative ends is given by Proposition \ref{bottomindex} with $n_3=n_4=0$. After summing we end up with
\begin{eqnarray}
\label{eqn:componentcalculation}
\mathrm{index}(C) \ge 2(P- \sum_p k^p)(N-3)+2\sum_p(n^p_2-k^p\tilde{n}^p_2)  \\
 + 2(N-2)(\sum_p k^p\tilde{n}^p_3-N_3) +2(N-2)(\sum_p k^p \tilde{n}^p_4-N_4)  \nonumber \\
- 2 \lfloor \frac{t^1_1}{b_n+\epsilon} \rfloor + 2\sum_{i,p} k^p\lfloor \frac{\tilde{t}^p_i}{b_n+\epsilon} \rfloor + 2\sum_{i,p} k^p\lfloor \tilde{u}^p_i(b_n+\epsilon) \rfloor \nonumber \\
+ 2Q(N-3) +2(N_3-1) + 2N_4 + 2\sum_{(i,p) \neq (1,1)} t^p_i + 2\sum_{i,p} u^p_i .  \nonumber
\end{eqnarray}

The last line in \eqref{eqn:componentcalculation} corresponds to terms in the index formulas for lower level curves which do not immediately cancel with terms in \eqref{eqn:indexestimate}.

We get a rougher estimate by ignoring all nonnegative terms corresponding to the matching ends. Suppose that our unmatched end covers the end corresponding to $\tilde{t}^1_1$ on $\tilde{u}^1$. Gathering multiples of $(N-3)$ this results in
\begin{eqnarray}
\label{eqn:componentcalculation2}
\frac{1}{2}\mathrm{index}(C) \ge (N-3)(P+Q- \sum_p k^p + \sum_p k^p\tilde{n}^p_3 -N_3 +\sum_p k^p \tilde{n}^p_4 - N_4 ) \\ +\sum_p(n^p_2-k^p\tilde{n}^p_2)
 + (\sum_p k^p\tilde{n}^p_3-N_3) + (\sum_p k^p \tilde{n}^p_4-N_4)
-  \lfloor \frac{t^1_1}{b_n+\epsilon} \rfloor + k^1 \lfloor \frac{\tilde{t}^1_1}{b_n+\epsilon} \rfloor.  \nonumber
\end{eqnarray}

Note that we have equality in the above formula only if there are no matching ends in our component.
Using the identity $P+Q = N_3 + N_4$ and removing more nonnegative terms (in particular the $(N-3)$ factor) we get
\begin{eqnarray}
\label{eqn:componentcalculation3}
\frac{1}{2}\mathrm{index}(C)\ge (N-3)(-\sum_p k^p + \sum_p k^p\tilde{n}^p_3 +\sum_p k^p \tilde{n}^p_4)
 +\sum_p(n^p_2-k^p\tilde{n}^p_2) \\
  + \sum_p (k^p\tilde{n}^p_3-n^p_3)
-  \lfloor \frac{t^1_1}{b_n+\epsilon} \rfloor + k^1 \lfloor \frac{\tilde{t}^1_1}{b_n+\epsilon} \rfloor  \nonumber \\
\ge \sum_p(n^p_2-k^p\tilde{n}^p_2) + k^1\tilde{n}^1_3-n^1_3
-  \lfloor \frac{t^1_1}{b_n+\epsilon} \rfloor + k^1 \lfloor \frac{\tilde{t}^1_1}{b_n+\epsilon} \rfloor.  \nonumber
\end{eqnarray}

Suppose the unmatched end of $C$ locally covers $\tilde{u}^1$ with degree $l \le k^1$, that is, $t^1_1=l\tilde{t}^1_1$. Then
using the inequality $l \lfloor \frac{x}{l}  \rfloor - \lfloor x \rfloor \ge -l+1$ (which follows, for example, from ``Hermite's identity") we have that $k^1 \lfloor \frac{\tilde{t}^1_1}{b_n+\epsilon} \rfloor - \lfloor \frac{t^1_1}{b_n+\epsilon} \rfloor \ge -l+1$. On the other hand $k^1\tilde{n}^1_3-n^1_3$ is at least $k^1$ minus the number of ends covering the first end of $\tilde{u}^1$, which is at most $1+(k^1-l)$. We conclude that
\begin{equation*}
\label{eqn:conclude}
\mathrm{index}(C) \ge 2\sum_p(n^p_2-k^p\tilde{n}^p_2)
\end{equation*}
with equality only if there are no matching ends and the proof is complete.
\end{proof}

{\em Step 3.} {\em Completion of the proof.}

Note that although the term $\sum_p(n^p_2-k^p\tilde{n}^p_2)$ in Lemma \ref{keyest} could be negative, it is bounded from below by
\begin{equation}
\label{eqn:boundedbelow}
-\sum_{i,p}(s^p_i-1),
\end{equation}
where we recall that the $s^p_i$ are the covering numbers of the limits of our component on $\alpha_2$. There is equality here if and only if all $\tilde{s}^p_i=1$.

We can now complete the proof of Theorem~\ref{compact}.

The index formulae and matching conditions in the compactness theorem imply that the sum of the indices of the limiting components, minus $2(N-2)m$, where $m$ is the number of ends matched on $\partial E(c, c+\epsilon) \times (\mathbb{C}P^1)^{N-2}$, is $0$, the index of curves in $\mathcal M$.  Note that the $2(N-2)m$ term here comes from the fact that the index formula from Lemma~\ref{topid} is for curves whose ends are allowed to vary in the corresponding Morse-Bott family.

Given our various bounds, this is only possible if all inequalities coming from Lemmas \ref{topid} and \ref{bottomid} and \ref{keyest} are equalities.  Specifically, the negative term for any component in $\overline{M}$ from \eqref{eqn:boundedbelow} must be precisely compensated for by the terms $2(c-1)$ in the index bound for curves in $\partial E(c, c+\epsilon) \times (\mathbb{C}P^1)^{N-2} \times \R$ from Lemma~\ref{topid}.  By Lemma \ref{topid}, equality in the index then gives that the components in $\partial E(c, c+\epsilon) \times (\mathbb{C}P^1)^{N-2}  \times \R$ are covers of trivial cylinders. Trivial cylinders have action $0$ and so we can also see that there was no gluing of lower level curves (which necessarily have positive area) to construct these components.

Also to have equality, Lemma \ref{keyest} showed that there was no matching to construct components in $\overline{M}$, and we see by Lemma~\ref{bottomid} that any components in $\partial E(1,b_n+\epsilon, T, \dots ,T) \times \R$ must be trivial cylinders.
In conclusion the only limiting curve of nonzero action is a single curve in $\overline{M}$.

To finish the proof we need to show that our limiting curve in $\overline{M}$ is somewhere injective. This will imply that it has nonnegative index and hence by Lemma \ref{topid} that the curves in $\partial E(c, c+\epsilon) \times (\mathbb{C}P^1)^{N-2}$ are trivial cylinders. Therefore we have compactness as required.

Assume to the contrary that this curve is a multiple cover of degree $k$ of some underlying curve $\tilde{u}$. For index equality, we have seen that the positive ends of $\tilde{u}$ must be simply covered, and so there are exactly $g_{n+1}/k$ such ends. There is a single negative end asymptotic to $\beta_1^{g_{n+2}/k}$. By our identity $3g_{n+1}=g_{n+2} + g_n,$ we then see then that $k$ actually divides all $g_n$. This is a contradiction.
\end{proof}

Note that the same argument gives that the moduli space $\mathcal{M}_{J}(g_{n+1}\alpha_2,\beta_1^{g_{n+2}})$ from Proposition~\ref{nontriv} is compact, as promised.

\end{section}

\begin{section}{Proof of Theorem \ref{main1}} \label{secn5}

\begin{proof}

Suppose there exists a symplectic embedding
\begin{equation}
\label{eqn:potentialembedding}
\lambda E(1,b_n+\epsilon,S, \dots ,S) \hookrightarrow B^4(c)\times \C^{N-2} \subset E(c,c+\epsilon) \times \C^{N-2}.
\end{equation}
Here $c=\frac{g_{n+2}}{g_{n+1}}$, $\lambda>0$ and $S$ is chosen sufficiently large that the moduli spaces described in section \ref{secn4} all have dimension $0$ or $1$ as claimed.

Arguing as in \cite{hk}, Lemma $3.1$, there exists a smooth family of symplectic embeddings $$\phi_t: \lambda(t)E(1,b_n+\epsilon,S, \dots ,S) \hookrightarrow E(c,c+\epsilon) \times \C^{N-2}$$ where $\lambda(0)=1$ and $\phi_0$ is a product embedding as discussed in section \ref{secn3}, and $\lambda(1)=\lambda$ and $\phi_1 = \phi$ is the embedding \eqref{eqn:potentialembedding}.  By slightly enlarging $c$ to $g_{n+2}/g_{n+1} + \epsilon'$ if necessary, we can assume that $\phi_t$ has image in $\op{int}(E(c,c+\epsilon)) \times \mathbb{C}^{N-2}$, and by choosing $T$ sufficiently large, in some $\op{int}(E(c,c+\epsilon)) \times \mathbb{C}P^1(2T)^{N-2}$.

Associated to these embeddings is a smooth family of almost complex structures $J_t$ on the corresponding completions of $E(c,c+\epsilon) \times \mathbb{C}P^1(2T)^{N-2} \setminus \phi_t(\lambda(t)E(1,b_n+\epsilon,S, \dots ,S))$.  We can view this family as a family of almost complex structures on $\overline{M}$, and we can assume without loss of generality that all these almost complex structures are equal outside of a compact set.  We can then consider the universal moduli space ${\cal M} = \{([u],t)| [u] \in {\cal M}^s(J_t), t \in [0,1]\}$ as in Theorem~\ref{compact}. By Proposition \ref{jexist} we may choose $J_0$ as in Proposition \ref{nontriv} so that ${\cal M}^s(J_0)$ represents a nontrivial cobordism class. Thus, by Theorem~\ref{compact}, ${\cal M}^s(J_1)$ is also nontrivial.

The action of curves in ${\cal M}(J_1)$ is $$g_{n+1}(c+\epsilon) - \lambda g_{n+2} = g_{n+2}(1-\lambda)+(\epsilon+\epsilon')g_{n+1}.$$ Therefore since holomorphic curves have positive action and $\epsilon$ and $\epsilon'$ can be chosen arbitrarily small we see that $\lambda \le 1$. This then implies Theorem \ref{main1} by Lemma~\ref{lem:mcslemma}.

\end{proof}

\end{section}

\end{document}